\documentclass[a4paper,12pt]{article}

\usepackage{amsmath,amsfonts,amssymb,lmodern,geometry,enumerate}
\usepackage[T1]{fontenc}
\usepackage[english]{babel}
\usepackage{dsfont}
\usepackage{amsthm}
\usepackage{color}
\usepackage{bbm}
\usepackage{mathrsfs,url,color}
\usepackage{textcase}

\newtheorem{thm}{Theorem}[section]
\newtheorem{prop}[thm]{Proposition}
\newtheorem{defi}[thm]{Definition}

\newtheorem{cor}[thm]{Corollary}

\makeatletter

\renewcommand\section{\@startsection {section}{1}{\z@}%
{-3.5ex \@plus -1ex \@minus -.2ex}%
{1.3ex \@plus.2ex}%
{\center\small\sc\MakeTextUppercase}}

\def\subsection#1{\@startsection {subsection}{2}{0pt}%
{-3.5ex \@plus -1ex \@minus -.2ex}%
{1ex \@plus.2ex}%
{\bf\mathversion{bold}}{#1}}

\def\subsubsection#1{\@startsection{subsubsection}{3}{0pt}%
{\medskipamount}%
{-10pt}%
{\normalsize\itshape}{\kern-2.2ex. #1.}}

\makeatother

\newcommand{\law}{\mathscr{L}}

\newcommand{\D}{\mathbb{D}}
 
\newcommand{\E}{\mathbb{E}}
\newcommand{\R}{\mathbb{R}}
\newcommand{\N}{\mathbb{N}}
\newcommand{\Z}{\mathbb{Z}}

\newcommand{\LL}{\textbf{L}}
\newcommand{\Var}{\mathrm{Var}}

\numberwithin{equation}{section}

\begin{document}

\title{\sc\bf\large\MakeUppercase{
Stein meets Malliavin in normal approximation
}}

\author{\sc Louis H. Y. Chen \\ \it National University of Singapore}

\def\parsedate #1:20#2#3#4#5#6#7#8\empty{20#2#3-#4#5-#6#7}
\def\moddate{\expandafter\parsedate\pdffilemoddate{\jobname.tex}\empty}
\date{\moddate}

\maketitle

\abstract

Stein's method is a method of probability approximation which hinges on the solution of a functional equation.  For normal approximation the functional equation is a first order differential equation. Malliavin calculus is an infinite-dimensional differential calculus whose operators act on functionals of general Gaussian processes.  Nourdin and Peccati (2009) established a fundamental connection between Stein's method for normal approximation and Malliavin calculus through integration by parts.  This connection is exploited to obtain error bounds in total variation in central limit theorems for functionals of general Gaussian processes. Of particular interest is the fourth moment theorem which provides error bounds of the order $\sqrt{\mathbb{E}(F_n^4)-3}$ in the central limit theorem for elements $\{F_n\}_{n\ge 1}$ of Wiener chaos of any fixed order such that $\mathbb{E}(F_n^2) = 1$. This paper is an exposition of the work of Nourdin and Peccati with a brief introduction to Stein's method and Malliavin calculus. It is based on a lecture delivered at the Annual Meeting of the Vietnam Institute for Advanced Study in Mathematics in July 2014.

\section{Introduction}

Stein's method was invented by Charles Stein in the 1960's when he used his own approach in class to prove a combinatorial central limit theorem of Wald and Wolfowitz \cite{W44} and of Hoeffding \cite{H51}. Malliavin calculus was developed by Paul Malliavin \cite{M78} in 1976 to provide a probabilistic proof of the H\"{o}rmander criterion (H\"{o}rmander \cite{H67}) of hypoellipticity. Although the initial goals of Stein's method and Malliavin calculus are different, they are both built on some integration by parts techniques. This connection was exploited by Nourdin and Peccati \cite{NP09} to develop a theory of normal approximation on infinite-dimensional Gaussian spaces. They were motivated by a remarkable discovery of Nualart and Peccati \cite{NP05}, who proved that a sequence of random variables in a Wiener chaos of a fixed order converges in distribution to a Gaussian random variable if and only of their second and fourth moments converge to the corresponding moments of the limiting random variable. By combining Stein's method and Malliavin calculus, Nourdin and Peccati \cite{NP09} obtained a general total variation bound in the normal approximation for functionals of Gaussian processes. They also proved that for $\{F_n\}$ in a Wiener chaos of fixed order such that $\E (F_n^2) = 1$, the error bound is of the order $\sqrt{\E(F_n^4) - 3}$, thus providing an elegant rate of convergence for the remarkable result of Nualart and Peccati \cite{NP05}. We call this result of Nourdin and Peccati \cite{NP09} {\it{the fourth moment theorem}}.
\medskip

The work of Nourdin and Peccati \cite{NP09} has added a new dimension to Stein's method. Their approach of combining Stein's method with Malliavin calculus has led to improvements and refinements of many results  in probability theory, such as the Breuer-Major theorem \cite{BM83}. More recently, this approach has been successfully used to obtain central limit theorems in stochastic geometry, stochastic calculus, statistical physics, and for zeros of random polynomials. It has also been extended to different settings as in non-commutative probability and Poisson chaos. {Of particular interest is the connection between the Nourdin-Peccati analysis and information theory, which was recently revealed in Ledoux, Nourdin and Peccati \cite{Le15} and in Nourdin, Peccati and Swan \cite{NOS13}. 
\medskip

This paper is an exposition on the connection between Stein's method and Malliavin calculus and on how this connection is exploited to obtain a general error bound in the normal approximation for functionals of Gaussian processes, leading to the proof of the fourth moment theorem with some applications. It is an expanded version of the first four sections and of part of section 5 of Chen and Poly \cite{CP14}, with most parts rewritten and new subsections added.

\section{Stein's method}
\subsection{A general framework}

Stein's method is a method of probability approximation introduced by
Charles Stein \cite{S72} in 1972. It does not involve Fourier
analysis but hinges on the solution of a functional equation. Although Stein's 1972 paper was on normal approximation, his ideas were general and applicable to other probability approximations. 

In a nutshell, Stein's method can be described as follows. Let $W$ and $Z$ be random elements taking values in a space $\mathcal{S}$ and let $\mathcal{X}$ and $\mathcal{Y}$ be
some classes of real-valued functions defined on $\mathcal{S}$. In approximating the distribution $\law(W)$ of $W$ by the distribution $\law(Z)$ of $Z$, we
write $\E h(W) -  \E h(Z) = \E Lf_h(W)$ for a test function $h\in \mathcal{Y}$,
where $L$ is a linear operator ({\it Stein operator}) from $\mathcal{X}$ into $\mathcal{Y}$ and $f_h \in \mathcal{X}$ a
solution of the equation 
\begin{equation}
  Lf = h - \E h(Z)\qquad {\text{({\it Stein equation})}.}
\end{equation}

The error
$\E Lf_h(W)$ can then be bounded by studying the solution $f_h$ and
exploiting the probabilistic properties of $W$. The operator $L$ characterizes $\law(Z)$  in the sense that $\law(W) = \law(Z)$ if and only if for a
sufficiently large class of functions $f$ we have
\begin{equation}
  \E Lf(W) = 0 \qquad {\text{({\it Stein identity})}.}
\end{equation}

In normal approximation, where $\law(Z)$ is the standard normal distribution, the operator used by Stein \cite{S72} is given by $Lf(w) = f'(w) - wf(w)$ for $w\in \mathbb{R}$, and in Poisson approximation, where $\law(Z)$ is the Poisson distribution with mean $\lambda > 0$, the operator $L$ used by Chen \cite{C75} is given by $Lf(w) = \lambda f(w + 1) - wf(w)$ for $w \in
\mathbb{Z}_+$.  However the operator $L$ is not unique even for the
same approximating distribution but depends on the problem at hand. For
example, for normal approximation $L$ can also be taken to be the generator of the Ornstein-Uhlenbeck process, that is, $Lf(w) = f''(w) - wf'(w)$, and for Poisson approximation, $L$ taken to be the generator of an immigration-death process, that is, $Lf(w) = \lambda[f(w + 1) - f(w)] + w[f(w - 1) - f(w)]$. This generator approach, which is due to Barbour \cite{B88}, allows
extensions to multivariate and process settings. Indeed, for multivariate normal approximation, $Lf(w) = \Delta f(w) - w\cdot
\nabla f(w)$, where $f$ is defined on the Euclidean space; see Barbour \cite{B90} and G\"otze \cite{G91}.
\medskip

Examples of expository articles and books on Stein's method for normal, Poisson and other probability approximations are Arratia, Goldstein and Gordon \cite{A90}, Chatterjee, Diaconis, and Meckes \cite{C05}, Barbour and Chen \cite{BC05a}, Barbour, Holst and Janson \cite{BHJ92}, Chen, Goldstein and Shao \cite{C11}, Chen and R\"ollin \cite{C13a}, Diaconis and Holmes \cite{D04}, and Ross \cite{R11}.

\subsection{Normal approximation}

In his 1986 monograh \cite{S86}, Stein proved the following characterization of the normal distribution.

\begin{prop} \label{prop-characterization} The following are equivalent.\\
(i) $W \sim \mathcal{N}(0,1)$;\\
(ii) $\E[f'(W) - Wf(W)] = 0$  for all $f \in \mathcal{C}_B^1$. 
\end{prop}

\begin{proof}
By integration by parts, (i) implies (ii). If (ii) holds, solve
\begin{equation} \label{Stein-equation}
f'(w) - wf(w) = h(w) - \E h(Z)
\end{equation}
where $h\in \mathcal{C}_B$ and $Z \sim \mathcal{N}(0,1)$.  Its solution $f_h$ is given by
\begin{eqnarray} \label{Stein-eqn-soln} \notag
f_h(w) &=& - e^{\frac{1}{2}w^2}\int_w^\infty e^{-\frac{1}{2}t^2}[h(t)-\E h(Z)]dt \\ 
&=& e^{\frac{1}{2}w^2}\int_{-\infty}^w e^{-\frac{1}{2}t^2}[h(t)-\E h(Z)]dt.
\end{eqnarray}
Using $\int_w^\infty e^{-\frac{1}{2}t^2}dt \le w^{-1}e^{-\frac{1}{2}w^2}$ for $w > 0$, we can show that $f_h \in \mathcal{C}_B^1$ with $\|f_h\|_\infty \le \sqrt{2\pi e}\|h\|_\infty$ and $\|f'_h\|_\infty \le 4\|h\|_\infty$. Substituting $f_h$ for $f$ in (ii) leads to
$$
\E h(W) = \E h(Z) \quad \text{for} \quad h \in \mathcal{C}_B.
$$
This proves (i).
\end{proof}

The proof of Propostion \ref{prop-characterization} shows that the Stein operator $L$ for normal approximation, which is given by $Lf(w) = f'(w) - wf(w)$, is obtained by integration by parts.

Assume $\E W = 0$ and $\mathrm{Var}(W)= B^2 >0$. By Fubini's theorem, for $f $ absolutely continuous for which the expectations exist, we have

\begin{equation*}
\E Wf(W) = \int_{-\infty}^\infty f'(x)\E WI(W > x)dx = B^2\E f'(W^*)
\end{equation*}
where $\law(W^*)$ is absolutely continuous with density given by $B^{-2}\E WI(W>x)$. The distribution $\law(W^*)$ is called {\it $W$-zero-biased}. The notion of zero-biased distribution was introduced by Goldstein and Reinert \cite{GR97}. 

Now assume $\mathrm{Var}(W)= 1$. By Propoition \ref{prop-characterization}, $\law(W) = \mathcal{N}(0,1)$ if and only if $\law(W^*) = \law(W)$. Heuristically this suggests that $\law(W^*)$ is "close" to $\law(W)$ if and only if $\law(W)$ is "close" to $\mathcal{N}(0,1)$. Therefore, it is natural to ask if can we couple $W^*$ with $W$ in such a way that $\E|W^* - W|$ provides a good measure of the  distance between $\law(W)$ and $\mathcal{N}(0,1)$? There are three distances commonly used for normal approximation.

\begin{defi}\label{def-distances}
Let $Z \sim \mathcal{N}(0,1)$, $F(x)=P(W\le x)$ and $\Phi(x)=P(Z \le x)$. \\
(i) The Wasserstein distance between $\law(W)$ and $\mathcal{N}(0,1)$ is defined by
$$
d_{\mathrm{W}}(\law(W), \mathcal{N}(0,1)):=\sup_{|h(x)-h(y)|\le |x-y|}|\E h(W)-\E h(Z)|.
$$
(ii) The Kolmogorov distance between $\law(W)$ and $\mathcal{N}(0,1)$ is defined by
$$
d_{\mathrm{K}}(\law(W), \mathcal{N}(0,1)) := \sup_{x \in \R}|F(x) - \Phi(x)|.
$$
(iii) The total variation distance between $\law(W)$ and $\mathcal{N}(0,1)$ is defined by
\begin{eqnarray*}
d_\mathrm{TV}(\law(W),\mathcal{N}(0,1))&:=& \sup_{A \in \mathcal{B}(\R)}|P(W\in A) - P(Z\in A)| \\
&=& \frac{1}{2}\sup_{|h|\le 1}|\E h(W)- \E h(Z)|.
\end{eqnarray*}
\end{defi}

Note that $\E h(W) - \E h(Z) = \E(h(W) - h(0)) - \E(h(Z) - h(0))$. So
$$
d_{\mathrm{W}}(\law(W), \mathcal{N}(0,1)) = \sup_{|h(x)-h(y)|\le |x-y|, h(0) =0}|\E h(W)-\E h(Z)|.
$$
Also note that $|h(w)| \le |h(w)-h(0)|+|h(0)|\le |w|+|h(0)|$. So $|h(x) - h(y| \le |x-y|$ implies that $h$ grows linearly. Since $\mathcal{C}^1$ functions $h$ with $\|h'\|_\infty \le 1$ is dense in the sup norm in the class of functions $h$ with $|h(x)-h(y)| \le |x-y|$, we also have
$$
d_{\mathrm{W}}(\law(W), \mathcal{N}(0,1))=\sup_{h \in \mathcal{C}^1, \|h'\|_\infty \le~1}|\E h(W)-\E h(Z)|.
$$
By an application of Lusin's theorem,
$$
\sup_{|h|\le 1}|\E h(W) - \E h(Z)| = \sup_{h \in \mathcal{C}, |h| \le 1} |\E h(W) - \E h(Z)|.
$$
Therefore,
$$
d_\mathrm{TV}(\law(W),\mathcal{N}(0,1))= \frac{1}{2}\sup_{h \in \mathcal{C}, |h| \le 1} |\E h(W) - \E h(Z)|.
$$

The proposition below concerns the boundedness properties of the solution $f_h$, given by (\ref{Stein-eqn-soln}), of the Stein equation (\ref{Stein-equation}) for $h$ either bounded or absolutely continuous with bounded $h'$. The use of these boundedness properties is crucial for bounding the distances defined in Definition \ref{def-distances}. 

\begin{prop} \label{prop-Stein-eqn-soln}
Let $f_h$ be the unique solution, given by (\ref{Stein-eqn-soln}), of the Stein equation (\ref{Stein-equation}), where $h$ is either bounded or absolutely continuous.

1. If $h$ is bounded, then
\begin{equation}\label{bd-bounded}
\|f_h\|_{\infty} \le \sqrt {2\pi}\|h\|_{\infty}, ~~\|f'_h\|_{\infty} \le 
4\|h\|_{\infty}.
\end{equation}

2. If $h$ is absolutely continuous with bounded $h'$, then
\begin{equation}\label{bd-abs-cont}
\|f_h\|_{\infty} \le 2\|h'\|_{\infty},~~\|f'_h\|_{\infty} \le \sqrt {2/\pi}\|h'\|_{\infty}, ~~\|f''_h\|_{\infty} \le 2\|h'\|_{\infty}.
\end{equation}

3. If $h = I_{(-\infty,x]}$ where $x \in \R$, then, writing $f_h$ as $f_x$, 
\begin{equation}\label{bd-indicator-1}
0 < f_x(w) \le  \sqrt{2\pi}/4, ~~|wf_x(w)| \le 1, ~~|f_x'(w)| \le 1, 
\end{equation}
and for all $w, u, v \in \R$, 
\begin{equation} \label{bd-indicator-2}
|f_x'(w) - f_x'(v)| \le 1,
\end{equation}
\begin{equation} \label{indicator-bd-4}
|(w+u)f_x(w+u)-(w+v)f_x(w+v)|\le(|w|+\sqrt{2\pi}/4)(|u|+|v|).
\end{equation}
\end{prop}

The bounds in the proposition and their proofs can be found in Lemmas 2.3 and 2.4 of Chen, Goldstein and Shao \cite{C11}.

In the case where $W^*$ is be coupled with $W$, that is, there is a {\it zero-bias coupling}, we have the following result.

\begin{thm} \label{thm-bd-zero-bias-coupling}
Assume that $\E W=0$ and $\Var(W)=1$ and that $W^*$ and $W$ are defined on the same probability space. Then
\begin{equation} \label{bd-zero-bias-couplng}
d_{\mathrm{W}}(\law(W), \mathcal{N}(0,1)) \le 2\E|W^*-W|.
\end{equation}
\end{thm}
\begin{proof}
Let $h$ be absolutely continuous with $\|h'\|_\infty \le 1$. Then by the definition of zero-biased distribution and by (\ref{bd-abs-cont}), 
\begin{eqnarray*}
|\E h(W)- \E h(Z)| &=& |\E[f_h'(W) - Wf_h(W)]| =|\E[f_h'(W)- f_h'(W^*)]  \\
&=& |\E\int_0^{W^*-W}f_h''(W+t)dt| \le \|f_h''\|_\infty \E|W^*-W| \\
&\le& 2\|h'\|_\infty\E|W^*-W| \le 2\E|W^*-W|.
\end{eqnarray*}
This proves the theorem.
\end{proof}
Theorem \ref{thm-bd-zero-bias-coupling} shows that $\E|W^*-W|$ provides an upper bound on the Wasserstein distance. We now construct a zero-bias coupling in the case where $W$ is a sum of independent random variables and show that $\E|W^*-W|$ indeed gives an optimal bound.

Let $X_1, \dots,X_n$ be independent random variables with $\E X_i=0$, $\mathrm{Var}(X_i)=\sigma_i^2 >0$ and $\E|X_i|^3 < \infty$. Let $W = \sum_{i=1}^n X_i$ and $W^{(i)}=W-X_i$. Assume $\Var(W)=1 (\Longrightarrow \sum_{i=1}^n \sigma_i^2 = 1)$.  
\medskip
\\
Define \\
(i) $I$ to be such that $P(I=i) = \sigma_i^2$ for $i=1,\cdots,n$;\\
(ii) $X_i^*$ to be $X_i$-zero-biased, $i = 1,\cdots,n$;\\
(iii) $I$, $X_1^*,\cdots,X_n^*$, $X_1,\cdots,X_n$ to be independent.
\medskip
\\
Then for absolutely continuous $f$ such that $\|f\|_\infty < \infty$ and $\|f'\|_\infty < \infty$,
\begin{eqnarray} \label{zero-bias coupling} \notag
\E Wf(W)&=& \sum_{i=1}^n \E X_if(W^{(i)}+ X_i) = \sum_{i=1}^n \sigma_i^2\E f'(W^{(i)}+X_i^*)\\ \label{W^*-coupled-indep-rv}
&=& \E f'(W^{(I)}+X_I^*)= \E f'(W^*).
\end{eqnarray}
So $W^*$ is coupled with $W$ and $W^*-W = X_I^* - X_I$. Note that the density of $X_i^*$ is given by $\sigma_i^{-2}\E X_iI(X_i > x)$. Straightforward calculations yield ${\displaystyle \E|X_I|\le \sum_{i=1}^n \E|X_i|^3}$ and ${\displaystyle \E|X_I^*| \le \frac{1}{2}\sum_{i=1}^n\E|X_i|^3}$. Therefore
\begin{equation} \label{E|W^*-W|-bd}
\E|W^*-W| = \E|X_I^*-X_I| \le \E|X_I^*| + \E|X_I| \le \frac{3}{2}\sum_{i=1}^n\E|X_i|^3.
\end{equation}
We immediately have the following corollary of Theorem \ref{thm-bd-zero-bias-coupling}.
\begin{cor} \label{cor-Wasserstein-bd}
Let $X_1. \cdots, X_n$ be independent with $\E X_i = 0$, $\Var(X_i) = \sigma_i^2$ and  $\E|X_i|^3  < \infty$, $i=1,\cdots,n$. Let $W = \sum_{i=1}^n X_i$ and assume that $\Var(W) = 1$. Then 
\begin{equation} \label{Wasserstein-bd}
d_{\mathrm{W}}(\law(W), \mathcal{N}(0,1)) \le 3\sum_{i=1}^n \E|X_i|^3.
\end{equation}
\end{cor}

It is much more difficult to obtain an optimal bound on the Kolmogorov distance between $\law (W)$ and ${\mathcal{N}(0,1)}$. Such a bound can be obtained by induction or by the use of a concentration inequality. For induction, see Bolthausen \cite{B84}. For the use of a concentration inequality, see Chen \cite{C98} and Chen and Shao \cite{CS01}  for sums of independent random variables, and Chen and Shao \cite{CS04} for sums of locally dependent random variables. See also Chen, Goldstein and Shao \cite{C11}. For sums of independent random variables, Chen and Shao \cite{CS01} obtained a bound of $4.1\sum \E|X_i|^3$ on the Kolmogorov distance. In the next subsection we will give a proof of an optimal bound on the Kolmogorov distance using the concentration inequality approach.

In general, it is difficult to construct zero-bias couplings such that $\E|W^*-W|$ is small for normal approximation. However, by other methods, one can construct an equation of the form,
\begin{equation}\label{Stein-general}
\E Wf(W) = \E T_1 f'(W + T_2),
\end{equation}
where $T_1$ and $T_2$ are some random variables defined on the same probability space as $W$, and $f$ is an absolutely continuous function for which the expectations in (\ref{Stein-general}) exist. Heuristically, in view of Proposition \ref{prop-characterization}, $\law(W)$ is "close" to $\mathcal{N}(0,1)$ if $T_1$ is "close" to $1$ and $T_2$ is "close" to $0$. Examples of $W$ satisfying this equation include sums of locally dependent random variables as considered in Chen and Shao \cite{CS04} and exchangeable pairs as defined in Stein \cite{S86}. More generally, a random variable $W$ satisfies (\ref{Stein-general}) if there is a {\it Stein coupling} $(W,W',G)$ where $W, W',G$ are defined on a common probability space such that $\E Wf(W)= \E\big[Gf(W') - Gf(W)\big]$ for absolutely continuous functions $f$ for which the expectations exist (see {Chen and R\"ollin \cite{C13b}}). In all cases it is assumed that $\E W = 0$ and $\Var (W) = 1$. Letting $f(w) = w$, we have $1 = \E W^2 = \E T_1.$ The case of zero-bias coupling corresponds to $T_1=1$. 
\medskip

As an illustration, let $(W,W')$ be an {\it exchangeable pair} of random variables, that is, $(W,W')$ has the same distribution as $(W',W)$. Assume that $\E W = 0$ and $\mathrm{Var}(W) = 1$ and that $\E \big[W' - W|W\big] = -\lambda W$ for some $\lambda > 0$. Since the function $(w,w') \longmapsto (w'-w)(f(w') + f(w))$ is anti-symmetric, the exchangeability of $(W,W')$ implies
$$
E\big[(W'-W)(f(W')+f(W))\big] = 0.
$$
From this we obtain
\begin{eqnarray*}
&&\E Wf(W) =\frac{1}{2\lambda} \E\Big[(W'-W)(f(W')-f(W))\Big] \\
&& \quad = \frac{1}{2\lambda}\E\Big[(W'-W)^2\int_0^1f'(W+ (W'-W)t)dt\Big] = {\E \big[T_1f'(W+T_2)\big]}
\end{eqnarray*}
where ${T_1 = {\displaystyle \frac{1}{2\lambda}(W'-W)^2}}$, ~$T_2 = (W'-W)U$, and $U$ uniformly distributed on $[0,1]$ and independent of ${W, W'}, ~T_1$ and $T_2$. The notion of exchangeable pair is central to Stein's method. It has been extensively used in the literature.
\medskip

Here is a simple example of an exchangeable pair. Let $X_1, \cdots,X_n$ be independent random variables such that $\E X_i = 0$ and $\Var(W)=1$, where $W = \sum_{i=1}^n X_i$. Let $X'_1,\cdots,X'_n$ be an independent copy of $X_1, \cdots,X_n$ and let $W' = W - X_I + X'_I$, where $I$ is uniformly distributed on $\{1, \cdots,n\}$ and independent of $\{X_i, X'_i, 1\le i \le n\}$. Then $(W,W')$ is an exchangeable pair and $\E[W'-W|W]= -\frac{1}{n}W$.
\medskip

Assume that $\E W = 0$ and $\Var (W) = 1$. From (\ref{Stein-equation}) and (\ref{Stein-general}),
\begin{eqnarray}\notag
\E h(W) - \E h(Z) &=& \E \big[f'_h(W) - T_1f'_h(W + T_2)\big]\\ \notag
&=&\E \big[T_1(f'_h(W) - f'_h(W + T_2))\big] + \E \big[(1 - T_1)f_h'(W)\big].\\
\label{Stein-solu-property}
\end{eqnarray}
Different techniques have been developed for bounding the error terms on the right side of (\ref{Stein-solu-property}). Apart from zero-bias coupling, which corresponds to $T_1=1$, we will focus on the case where $T_2 = 0$. This is the case if $W$ is a functional of independent Gaussian random variables as considered by Chatterjee \cite{C09} or a functional of Gaussian random fields as considered by Nourdin and Peccati \cite{NP09}. In this case, (\ref{Stein-solu-property}) becomes
$$
\E h(W) - \E h(Z) = \E \big[(1 - T_1)f_h'(W)\big] = \E\big[(1 - \E[T_1|W])f_h'(W)\big].
$$
Let $h$ be such that $|h| \le 1$. Then, by Proposition \ref{prop-Stein-eqn-soln}, we obtain the following bound on the  total variation distance between $\law(W)$ and $\mathcal{N}(0,1)$.
\begin{eqnarray*}
d_{\mathrm{TV}}(\law(W), \mathcal{N}(0,1)) &:=& \frac{1}{2}\sup_{|h|\le 1}|\E h(W) - \E h(Z)| \\
&\le& \frac{1}{2}\|f_h'\|_{\infty}\E|1 - \E[T_1|W]|\\ 
&\le& 2\|h\|_\infty\E|1 - \E[T_1|W]| \\
&\le&  2\sqrt{\Var(\E[T_1|W])},
\end{eqnarray*}
where for the last inequality it is assumed that $\E[T_1|W]$ is square integrable.
\medskip

While Chatterjee \cite{C09} developed second order Poincar\'e inequalities  to bound $2\sqrt{\Var(\E[T_1|W])}$, Nourdin and Peccati \cite{NP09} deployed Malliavin calculus. In Sections 3 and 4, we will discuss how Malliavin calculus is used to bound $2\sqrt{\Var(\E[T_1|W])}$.

\subsection{Berry-Esseen theorem}
In this subsection, we will give a proof of the Berry-Esseen theorem for sums of independent random variables using zero-bias coupling and a concentration inequality.

\begin{thm}[Berry-Esseen] \label{B-E-thm}
Let $X_1,\cdots,X_n$ be independent random variables with $\E X_i = 0$, $\Var(X_i)= \sigma^2_i$, and $\E|X_i|^3 = \gamma_i < \infty$. Let $W= \sum_{i=1}^n X_i$ and assume $\Var(W)=1$. Then
\begin{equation}\label{B-E-bd}
d_{\mathrm{K}}(\law(W), \mathcal{N}(0,1)) \le 7.1\sum_{i=1}^n \gamma_i.
\end{equation}
\end{thm}

We first prove two propositions using the same notation as in Theorem \ref{B-E-thm}. Let $W^*$ be $W$-zero-biased and assume that it is coupled with $W$ as given in (\ref{zero-bias coupling}). Let $\Phi$ denote the distribution function of $\mathcal{N}(0,1)$.

\begin{prop} \label{B-E-bd-zero-bias-prop}
For $x \in \R$,
\begin{eqnarray} \label{be-zero-bias-bd}
|P(W^*\le x) - \Phi(x)| \le 2.44\sum_{i=1}^n\gamma_i.
\end{eqnarray}
\end{prop}
\begin{proof}
Let $f_x$ be the unique bounded solution of the Stein equation
\begin{eqnarray} \label{Stein-eq-ci}
f'(w) - wf(w) = I(w \le x ) - \Phi(x)
\end{eqnarray}
where $x\in\R$.
The solution $f_x$ is given by (\ref{Stein-eqn-soln}) with $h(w) = I(w\le x)$.
From this equation and by (\ref{indicator-bd-4}),
\begin{eqnarray*} 
&&|P(W^* \le x) - \Phi(x)| \\
&& \quad = |\E[ f'_x(W^*)-W^*f_x(W^*)]|\\ 
&& \quad = |\E[(W^{(I)}+X_I)f_x(W^{(I)}+X_I)-(W^{(I)}+X_I^*)f_x(W^{(I)}+X_I^*)]| \\ 
&& \quad \le \E(|W^{(I)}|+ \frac{\sqrt{2\pi}}{4})(|X_I|+|X_I^*|)\\ 
&& \quad \le \frac{3}{2}(1+\frac{\sqrt{2\pi}}{4})\sum_{i=1}^n\gamma_i \le 2.44\sum_{i=1}^n\gamma_i.
\end{eqnarray*}
This proves Proposition \ref{B-E-bd-zero-bias-prop}.
\end{proof}

Next we prove a concentration inequality.

\begin{prop} \label{prop-ci}
For $i=1,\dots,n$ and for $a \le b$, $a, b \in \R$, we have
\begin{eqnarray} \label{ci}
P(a\le W^{(i)} \le b) \le \frac{2\sqrt{2}}{3}(b-a)+\frac{4(\sqrt{2}+1)}{3}\sum_{i=1}^n \gamma_i.
\end{eqnarray}
\end{prop}

\begin{proof}
This proof is a slight variation of that of Lemma 3.1 in Chen, Goldstein and Shao \cite{C11}. Let $\delta > 0$ and let $f$ be given by $f((a+b)/2) = 0$ and $f'(w) = I(a-\delta \le w \le b+\delta)$. Then $|f| \le (b-a+2\delta)/2$. 
Since $X_j$ is independent of $W^{(i)} - X_j$ for $j \neq i$,  $X_i$ is independent of $W^{(i)}$, and since $\E X_j = 0$ for $j=1,\dots,n$, we have
\begin{eqnarray} \notag \label{Stein-id-for-ci}
&&\E W^{(i)}f(W^{(i)})-\E X_if(W^{(i)}-X_i) \\ \notag
&& \quad = \sum_{j=1}^n \E X_j[f(W^{(i)}) - f(W^{(i)}-X_j)] \\ \notag
&& \quad = \sum_{j=1}^n\E X_j\int_{-X_j}^0 f'(W^{(i)}+t)dt \\ \notag
&& \quad = \sum_{j=1}^n\E X_j\int_{-X_j}^0 I(a-\delta \le W^{(i)} + t \le b + \delta)dt \\ \notag
&& \quad \ge \sum_{j=1}^n\E X_j\int_{-X_j}^0 I(a-\delta \le W^{(i)} + t \le b + \delta)I(|t| \le \delta)dt \\ \notag
&& \quad \ge \E I(a \le W^{(i)} \le b )\sum_{j=1}^n X_j\int_{-X_j}^0I(|t| \le \delta)dt \\ \notag
&& \quad = \E I(a \le W^{(i)} \le b )\sum_{j=1}^n |X_j|\min(|X_j|,\delta) \\ \notag
&& \quad \ge P(a \le W^{(i)} \le b )\sum_{j=1}^n \E|X_j|\min(|X_j|,\delta) \\ \notag 
&& \quad \quad -  \E I(a \le W^{(i)} \le b )\big|\sum_{j=1}^n [|X_j|\min(|X_j|,\delta)- \E|X_j|\min(|X_j|,\delta)]\big| \\ \label{bd-R_1-R_2-c-i}
&& \quad = R_1 - R_2,
\end{eqnarray}
where in the first inequality in (\ref{Stein-id-for-ci}), we used the fact that 
$$
X_j\int_{-X_j}^0 I(a-\delta \le W^{(i)} + t \le b + \delta)dt \ge 0 ~ \mathrm{for} ~ j = 1, \cdots,n.
$$
Using the inequality, ${\displaystyle \min(a,b) \ge a - \frac{a^2}{4b}}$ for $a,b>0$, we obtain
\begin{eqnarray} \notag
R_1 &\ge& P(a \le W^{(i)} \le b )\{\sum_{j=1}^n\E X_j^2 - \frac{1}{4\delta}\sum_{j=1}^n \E|X_j|^3\} \\ \label{bd-R_1-c-i}
&=& P(a \le W^{(i)} \le b )\{1 - \frac{1}{4\delta}\sum_{j=1}^n \E|X_j|^3\}.
\end{eqnarray}
We also have
\begin{eqnarray}\notag
R_2 &\le& \E \big|\sum_{j=1}^n [|X_j|\min(|X_j|,\delta)- \E|X_j|\min(|X_j|,\delta)]\big| \\ \notag
&\le& \big[\Var(\sum_{j=1}^n |X_j|\min(|X_j|,\delta))\big]^{1/2} \\ \notag
&\le& \big[\sum_{j=1}^n \E X_j^2\min(|X_j|,\delta)^2\big]^{1/2} \\ \label{bd-R_2-c-i}
&\le& \delta \big[\sum_{j=1}^n \E X_j^2\big]^{1/2} = \delta.
\end{eqnarray}
Bounding the left hand side of (\ref{bd-R_1-R_2-c-i}), we obtain
\begin{eqnarray}  \notag
&&\E W^{(i)}f(W^{(i)}) - \E X_if(W^{(i)}-X_i)  \\ \notag
&& \quad \le \frac{1}{2}(b-a+2\delta)(\E|W^{(i)}|+\E|X_i|) \qquad \qquad \qquad   \\ \notag
&& \quad \le \frac{1}{\sqrt{2}}(b-a+2\delta)[(\E|W^{(i)}|)^2+(\E|X_i|)^2]^{1/2} \\ \notag
&& \quad \le \frac{1}{\sqrt{2}}(b-a+2\delta)[\E(W^{(i)2})+\E(X_i^2)]^{1/2} \\ \label{bd-lhs-c-i}
&& \quad = \frac{1}{\sqrt{2}}(b-a+2\delta).
\end{eqnarray}
The proof of Proposition \ref{prop-ci} is completed by letting ${\displaystyle \delta = \sum_{j=1}^n\E|X_j|^3}$ and combining (\ref{bd-R_1-R_2-c-i}), (\ref{bd-R_1-c-i}), (\ref{bd-R_2-c-i}) and (\ref{bd-lhs-c-i}).
\end{proof}

We now prove Theorem \ref{B-E-thm}. By Proposition \ref{B-E-bd-zero-bias-prop}, we have
\begin{eqnarray} \label{B-E-ci-W-zero-bias} \notag
|P(W \le x) - \Phi(x)| &\le& |P(W\le x) - P(W^{(I)}+X_I^* \le x)| \\ \notag
&& + 2.44\sum_{i=1}^n\gamma_i \\ \notag
&=& \E I(x-X_I\vee X_I^* \le W^{(I)} \le x - X_I\wedge X_I^*) \\ 
&& + 2.44\sum_{i=1}^n\gamma_i.
\end{eqnarray}
Using the independence between $I$ and $\{X_i,X_i^*,1\le i \le n\}$, we have  
\begin{eqnarray*}
&&\E I(x-X_I\vee X_I^* \le W^{(I)} \le x - X_I\wedge X_I^*) \\
&& \quad = \sum_{i=1}^n\sigma_i^2\E I(x-X_i\vee X_i^* \le W^{(i)} \le x - X_i\wedge X_i^*) \\
&& \quad = \sum_{i=1}^n\sigma_i^2\E P(x-X_i\vee X_i^* \le W^{(i)} \le x - X_i\wedge X_i^*|X_i,X_i^*).
\end{eqnarray*}
Since $W^{(i)}$ is indpendent of $X_i,X_i^*$ for $i = 1, \cdots, n$, it follows from (\ref{ci}) that
\begin{eqnarray} \label{ci-W-zero-bias} \notag
&&\E I(x-X_I\vee X_I^* \le W^{(I)} \le x - X_I\wedge X_I^*) \\ \notag
&& \quad \le \sum_{i=1}^n\sigma_i^2\E\Big[\frac{2\sqrt{2}}{3}(|X_i|+|X_i^*|)+ \frac{4(\sqrt{2}+1)}{3}\sum_{i=1}^n\gamma_i\Big] \\ 
&& \quad \le \left(\sqrt{2}+ \frac{4(\sqrt{2}+1)}{3}\right)\sum_{i=1}^n\gamma_i
\le 4.65\sum_{i=1}^n\gamma_i.
\end{eqnarray} 
The proof of Theorem \ref{B-E-thm} is completed by combining (\ref{B-E-ci-W-zero-bias}) and (\ref{ci-W-zero-bias}).

\section{Malliavin calculus}
\subsection{Preamble}
In this paper, the work of Nourdin and Peccati will be presented in the context of the Gaussian process $X=\{\int_0^\infty f(t)dB_t : f \in L^2(\R_+)\}$, where $(B_t)_{t \in \R_+}$ is a standard Brownian motion on acomplete probability space $(\Omega,\mathcal{F}, P)$, where $\mathcal{F}$ is generated by  $(B_t)_{t \in \R_+}$, and $L^2(\R_+)$ is the separable Hilbert space of square integrable real-valued functions with respect to the Lebesgue measure on $\R_+$. This Gaussian process is a centered Gaussian family of random variables with the covariance given by 
$$
\E \big[\int_0^\infty f(t)dB_t\int_0^\infty g(t)dB_t \big] = \langle f,g \rangle_{L^2(\R_+)}.
$$ There will be no loss of generality since problems of interest are of distributional nature and through an isometry these problems can be transferred to~$X$. 
\medskip

More specifically, let $Y=\{Y(h): h \in \mathcal{H} \}$ be a centered Gaussian process over a real separable Hilbert space $\mathcal{H}$ with the covariance given by
$$
\E \big[Y(h_1) Y(h_2)\big] = \langle h_1,h_2 \rangle_\mathcal{H}.
$$
Let $\psi:\mathcal{H}\rightarrow L^2(\R_+)$ be an isometry and let $f_1=\psi(h_1)$ and $f_2=\psi(h_2)$ for $h_1,h_2 \in \mathcal{H}$. Then
$$
\E \big[\int_0^\infty f_1(t)dB_t\int_0^\infty f_2(t)dB_t \big] = \E[Y(h_1)Y(h_2)]. 
$$
This implies that $\law(X) = \law(Y)$ and problems of distributional nature on $Y$ can be transferred to $X$. 
\medskip

The material in this section can be found in Nourdin \cite{N12} and Nourdin and Peccati \cite{NP12}.

\subsection{Multiple Wiener-It\^o integrals and Wiener chaos}

Let $B = (B_t)_{t \in \R_+}$ be a standard Brownian motion on a complete probability space $(\Omega,\mathcal{F}, P)$, where $\mathcal{F}$ is generated by  $(B_t)_{t \in \R_+}$, and let $f \in L^2(\R_+^p)$ where $p$ is a positive integer. We define 
\begin{equation}\label{multiple W-I}
I_p(f) = \sum_{\sigma} \int_0^\infty dB_{t_1}\int_0^{t_1}dB_{t_2}\dots \int_0^{t_{p-1}}dB_{t_p}f(t_{\sigma(1)},t_{\sigma(2)},\dots,t_{\sigma(p)})
\end{equation}
where the sum is over all permutations $\sigma$ of $\{1,2,\dots,p\}$. The random variable $I_p(f)$ is called the {\it $p$th multiple Wiener-It\^o integral}. The closed linear subspace $\mathcal{H}_p$ of $L^2(\Omega)$ generated by $I_p(f)$, $f \in L^2(\R_+^p)$, is called the {\it $p$th Wiener chaos} of $B$. We use the convention that $\mathcal{H}_0 = \R$.
\medskip

If $f$ is symmetric, that is, $f(t_1,\dots,t_p) = f(t_{\sigma(1)},\dots,t_{\sigma(p)})$ for any permutation $\sigma$ of $\{1,\dots,p\}$, then
$$
I_p(f) = p! \int_0^\infty dB_{t_1}\int_0^{t_1}dB_{t_2}\dots \int_0^{t_{p-1}}dB_{t_p}f(t_1,t_2,\dots,t_p).
$$

We define the symmetrization of $f \in L^2(\R_+^p)$ by 
\begin{equation}\label{symmetrization}
\tilde{f}(t_1,\dots,t_p) = \frac{1}{p!}\sum_{\sigma}f(t_{\sigma(1},\dots,t_{\sigma(p)}),
\end{equation}
where the sum is over all permutations $\sigma$ of $\{1,\dots,p\}$. Let $L^2_s(\R^p_+)$ be the closed subspace of $L^2(\R^p_+)$ of symmetric functions. By the triangle inequality,
$$
\|\tilde{f}\|_{L^2(\R^p_+)} \le \|f\|_{L^2(\R^p_+)},
$$ 
we see that $f \in L^2(\R^p_+)$ implies $\tilde{f} \in L^2_s(\R^p_+)$. The following properties of the stochastic integrals $I_p(\cdot)$ can be easily verified: 
~\\
(i) $\E I_p(f) = 0$ and $I_p(f) = I_p(\tilde{f})$ for all $f \in L^2(\R^p_+)$.\\
(ii) For all $f \in L^2(\R^p_+)$ and $g \in L^2(\R^q_+)$,
\begin{equation}\label{ortho-relation} 
\E[I_p(f)I_q(g)] = 
\left\{\begin{array}{lrl}
0 & \text{for} & p \neq q,\\ 
p! \langle \tilde{f},\tilde{g}\rangle_{L^2(\R^p_+)}& \text{for} & p = q. \end{array}\right.
\end{equation} \\
(iii) The mapping $f \mapsto I_p(f)$ from $L^2(\R^p_+)$ to $L^2(\Omega)$ is linear.
\medskip

The multiple Wiener-It\^o integrals are infinite dimensional generalizations of the Hermite polynomials. The {\it $k$th Hermite polynomial} $H_k$ is defined by
$$
H_k(x)= (-1)^k e^{\frac {x^2}{2}} \frac {d^k} {dx^k} \left( e^{-\frac {x^2}{2}}\right), x \in \R.
$$

If $f \in L^2(\R_+)$ such that $\|f\|_{L^2(\R_+)} = 1$, it can be shown that
\begin{equation}\label{isometry1}
I_k(f^{\otimes k})=H_k\left(\int_0^\infty f(t) dB_t\right),
\end{equation}
where $f^{\otimes k} \in L^2(\R^k_+)$ is the $k$th tensor product of $f$ with itself defined by $f(t_1,\dots,t_k) = f(t_1)\dots f(t_k)$. If $\phi=f_1^{\otimes k_1}\otimes\cdots\otimes f_p^{\otimes k_p}$ with $(f_i)_{1\le i\le p}$ an orthonormal system in $L^2(\R_+)$ and $k_1+\cdots+k_p=k$, (\ref{isometry1}) can be extended to
\begin{equation}\label{isometry2}
I_k(\phi)=\prod_{i=1}^p H_{k_i}\left(\int_0^\infty f_i(t) dB_t\right).
\end{equation}

As in one-dimension where the Hermite polynomials form an orthogonal basis for $L^2(\R, \frac{1}{\sqrt{2\pi}}e^{-x^2/2}dx)$, the space $L^2(\Omega)$ can be decomposed into an infinite orthogonal sum of the closed subspaces $\mathcal{H}_p$. We state this fundamental fact about Gaussian spaces as a theorem below.
\begin{thm}
Any random variable $F\in L^2(\Omega)$ admits an orthogonal decomposition of the form
\begin{equation}\label{chaotic-expansion}
F= \sum_{k=0}^\infty I_k(f_k),
\end{equation}
where $I_0(f_0) = \E[F]$, and $f_k\in  L^2(\R^k_+)$ are symmetric and uniquely determined by $F$. 
\end{thm}
~\\
Applying the orthogonality relation (\ref{ortho-relation}) to the symmetric kernels $f_k$ for $F$ in the Wiener chaos expansion (\ref{chaotic-expansion}),
\begin{equation}\label{norm-expansion}
\|F\|^2_{L^2(\Omega)} =  \sum_{k=0}^\infty k!\|f_k\|^2_{L^2(\R^k_+)}.
\end{equation}

The random variables $I_k(f_k)$ inherit some properties from the algebraic structure of the Hermite polynomials, such as the product formula (\ref{prodfor}) below. To understand this we need the definition of  {\it{contraction}}.
 
\begin{defi}\label{deficontract}
Let $p, q \ge 1$ and let $f\in L^2(\R^p_+)$ and $g\in L^2(\R^q_+)$ be two symmetric functions. For $r \in \{1, \dots, p\wedge q\}$, the $r$th \textit{contraction} of $f$ and $g$, denoted by $f\otimes_r g$, is defined by

$$
f\otimes_r g (x_1,\dots,x_{p-r},y_1,\cdots,y_{q-r}) \qquad \qquad \qquad \qquad \qquad \qquad \qquad \qquad \qquad
$$
$$
= \int_{\R_+^r}f(x_1,\cdots,x_{p-r},t_1,\cdots,t_r)g(y_1,\cdots,y_{q-r},t_1,\cdots,t_r)dt_1\cdots dt_r.
$$
By convention, $f\otimes_0 g$ = $f\otimes g$.
\end{defi}

The contraction $f\otimes_r g$ is not necessarily symmetric, and we denote by $f \widetilde{\otimes}_rg$ its symmetrization. Note that by the Cauchy-Schwarz inequality, 
$$
\|f\otimes_r g\|_{L^2(\R^{p+q-2r}_+)} \le \|f\|_{L^2(\R^p_+)}\|g\|_{L^2(\R^q_+)} \quad \text{for} \quad  r = 0,1,\dots,p\wedge q,
$$
and that $f\otimes_p g = \langle f, g \rangle_{L^2(\R^p_+)}$ when $p=q$.
\medskip

We state the {\it{product formula}} between two multiple Wiener-It\^o integrals in the next theorem.
\begin{thm}
Let $p, q \ge 1$ and let $f\in L^2(\R^p_+)$ and $g\in L^2(\R^q_+)$ be two symmetric functions. Then
\begin{equation}\label{prodfor}
I_p(f) I_q(g)=\sum_{r=0}^{p\wedge q} r! \binom{p}{r} \binom{q}{r} I_{p+q-2r}(f\widetilde{\otimes}_r g).
\end{equation}

\end{thm}

\subsection{Malliavin derivatives}

Let $B = (B_t)_{t \in \R_+}$ be a standard Brownian motion on a complete probability space $(\Omega,\mathcal{F}, P)$, where $\mathcal{F}$ is generated by  $B$, and let $X=\{X(h), h\in L^2(\R_+)\}$ where
$X(h) = \int_0^\infty hdB_t$.  The set $X$ is a centered Gaussian family of random variables defined on $(\Omega, \mathcal{F},P)$, with covariance given by
$$
\E[X(h)X(g)]= \langle h,g \rangle_{L^2(\R_+)},
$$
for $h,g \in L^2(\R_+)$. Such a Gaussian family is called an \textit{isonormal Gaussian process} over $L^2(\R_+)$. 
\medskip

Let $\mathcal{S}$ be the set of all cylindrical random variables of the form:
\begin{equation} 
F=g\left( X(\phi _{1}),\ldots ,X(\phi _{n})\right) ,  \label{v3}
\end{equation}
where $n\geq 1$, $g:\mathbb{R}^{n}\rightarrow \mathbb{R}$ is an infinitely
differentiable function such that its partial derivatives have polynomial growth, and $\phi _{i}\in L^2(\R_+)$, $i~=~1,\ldots,n$. It can be shown that the set $\mathcal{S}$ is dense in $L^2(\Omega)$.
The {\it Malliavin derivative}  of $F \in \mathcal{S}$ with respect to $X$ is the element of
$L^2(\Omega ,L^2(\R_+))$ defined as

\begin{equation}\label{defi-malliavin}
DF\;=\;\sum_{i=1}^{n}\frac{\partial g}{\partial x_{i}}\left( X(\phi
_{1}),\ldots ,X(\phi _{n})\right) \phi _{i}.
\end{equation}

In particular, $DX(h)=h$ for every $h\in L^2(\R_+)$. By iteration, one can
define the $m$th derivative $D^{m}F$, which is an element of $L^2(\Omega ,L^2(\R^m_+))$ for every $m\geq 2$, as follows.

\begin{equation}\label{iterate-Malliavin}
D^m F=\sum_{i_1,\cdots,i_m=1}^n \frac{\partial^m g}{\partial x_{i_1}\cdots \partial x_{i_m}}\Big[X(\phi_1),\cdots,X(\phi_n)\Big] \phi_{i_1}\otimes \cdots \otimes \phi_{i_m}.
\end{equation}

The Hilbert space $L^2(\Omega ,L^2(\R^m_+))$ of $L^2(\R^m_+)$-valued functionals of $B$ is endowed with the inner product,
$$
\langle u,v \rangle_{L^2(\Omega,L^2(\R^m_+))} = \E\langle u,v \rangle_{L^2(\R^m_+).}
$$

For $m\geq 1$, it can be shown that $D^m$ is closable from  $\mathcal{S}$ to $L^2(\Omega, \R_+^m)$. So the domain of $D^m$ can be extended to ${\mathbb{D}}^{m,2}$,  the closure of $\mathcal{S}$ with respect to the norm $\Vert \cdot \Vert _{m,2}$, defined by

\begin{equation*}
\Vert F\Vert_{m,2}^{2}\;=\;\E\left[ F^{2}\right] +\sum_{i=1}^{m}\E\left[
\Vert D^{i}F\Vert _{L^2(\R^i_+)}^{2}\right].
\end{equation*}

A random variable $F \in L^2(\Omega)$ having the Wiener chaos expansion (\ref{chaotic-expansion}) is an element of ${\mathbb{D}}^{m,2}$ if and only if the kernels $f_k$, $k = 1,2, \dots$ satisfy
$$
\sum_{k=1}^\infty k^m k! \|f_k\|^2_{L^2(\R^k_+)} < \infty,
$$
in which case,
$$
\E\|D^m F\|^2_{L^2(\R^m_+)} = \sum_{k=m}^\infty (k)_m k! \|f_k\|^2_{L^2(\R^k_+)},
$$
where $(k)_m$ is the falling factorial. In particular, any $F$ having a finite Wiener chaos expansion is an element of ${\mathbb{D}}^{m,2}$ for all $m \ge 1$.
\medskip

The Malliavin derivative $D$, defined in (\ref{defi-malliavin}), obeys the following \textsl{chain rule}. If
$g:\mathbb{R}^{n}\rightarrow \mathbb{R}$ is continuously
differentiable with bounded partial derivatives and if $F=(F_{1},\ldots
,F_{n})$ is such that $F_i \in {\mathbb{D}}^{1,2}$ for $i=1,\dots,n$, then $g(F)\in {\mathbb{D}}^{1,2}$ and
\begin{equation}\label{e:chainrule}
Dg(F)=\sum_{i=1}^{n}\frac{\partial g }{\partial x_{i}}(F)DF_{i}.
\end{equation}

The domain $\D^{1,2}$ can be described in terms of the Wiener chaos decomposition as
\begin{equation}\label{Domain-D}
\D^{1,2}=\Big\{F\,\in L^2(\Omega)\,:\,\,\sum_{k=1}^{\infty }k\|I_k(f_k)\|^2_{L^2(\Omega)}<\infty\Big\}.
\end{equation}

The derivative of $F \in \mathbb{D}^{1,2}$, where $F$ is of the form (\ref{chaotic-expansion}), can be identified with the element of $L^2(\R_+ \times \Omega)$ given by
\begin{equation}\label{dtf}
D_{t}F=\sum_{k=1}^{\infty }kI_{k-1}\left( f_{k}(\cdot ,t)\right) ,\quad t \in \R_+.  
\end{equation}
Here $I_{k-1}(f_k(\cdot,t))$ denotes the Wiener-It\^{o} integral of order $k-1$ with respect to the $k-1$ remaining coordinates after holding $t$ fixed. Since the $f_k$ are symmetric, the choice of the coordinate held fixed does not matter. 
\medskip

The \textit{Ornstein-Uhlenbeck} operator $\LL$ is defined by the following relation
\begin{equation}\label{DiagoL}
\LL(F)=\sum_{k=0}^{\infty }-kI_k(f_k),
\end{equation}
for $F$ represented by (\ref{chaotic-expansion}). It expresses the fact that $\LL$ is diagonalizable with spectrum $-\N$ and the Wiener chaos as eigenspaces. The domain of $\LL$ is
\begin{equation}\label{domaineL}
\mathrm{Dom}(\LL)=\{F\in L^2(\Omega ):\sum_{k=1}^{\infty }k^{2}k!\|f_k\|
_{L^2(\R_+^k)}^{2}<\infty \}=\mathbb{D}^{2,2}\text{.}
\end{equation}
If $F=g(X(h_1),\cdots,X(h_n))$, where $g \in C^2(\R^n)$ with bounded first and second partial derivatives, it can be shown that
\begin{eqnarray}\label{Ornstein-uhlenbeck} \notag
\LL(F)&=&\sum_{i,j=1}^n \frac{\partial^2g}{\partial x_i \partial x_j}(X(h_1),\cdots,X(h_n))\langle h_i,h_j \rangle_{L^2(\R_+)} \\
&& \quad -\sum_{i=1}^n X(h_i)\frac{\partial g}{\partial x_{i}}(X(h_1),\cdots,X(h_n)).
\end{eqnarray}

The operator $\LL^{-1}$, which is called the {\it{pseudo-inverse}} of $\LL$, is defined as follows.
\begin{equation}\label{inverse-DiagoL}
\LL^{-1}(F) =\LL^{-1} \Big[ F-\E[F]\Big]=\sum_{k=1}^\infty -\frac {1}{k} I_k(f_k),
\end{equation}
for $F$ represented by (\ref{chaotic-expansion}). The domain of $\LL^{-1}$ is $\mathrm{Dom}(\LL^{-1}) = L^2(\Omega)$. It is obvious that for any $F \in L^2(\Omega)$, we have $\LL^{-1}F \in \mathbb{D}^{2,2}$ and 
\begin{equation} \label{LL-inverse}
\LL\LL^{-1}F = F - \E[F].
\end{equation}

A crucial property of $\LL$ is the following integration by parts formula. For $F \in \mathbb{D}^{2,2}$ and $G \in \mathbb{D}^{1,2}$, we have 
\begin{equation} \label{int-by-parts-L}
\E[\LL F \times G] = - \E[\langle DF,DG \rangle_{L^2(\R_+)}].
\end{equation}
By the bilinearity of the inner product and the Wiener chaos expansion (\ref{chaotic-expansion}), it suffices to prove (\ref{int-by-parts-L}) for $F = I_p(f)$ and $G = I_q(g)$ with $p,q \ge 1$ and $f \in L^2(\R_+^p)$, $g \in L^2(\R_+^q)$ symmetric. When $p \neq q$, we have
$$
\E[\LL F \times G] = -p\E[I_p(f)I_q(g)] = 0
$$
and
$$
\E[\langle DF,DG \rangle_{L^2(\R_+)}] = pq\int_0^\infty\E[I_{p-1}(f(\cdot,t))I_{q-1}(g(\cdot,t))]dt =0.
$$
So (\ref{int-by-parts-L}) holds in this case. When $p=q$, we have 
$$
\E[(\LL F)G] = -p\E[I_p(f)I_q(g)] = -p p!\langle f,g \rangle_{L^2(\R_+^p)}
$$
and
\begin{eqnarray*}
\E[\langle DF,DG\rangle_{L^2(\R_+)}]&=& 
p^2\int_0^\infty\E[I_{p-1}(f(\cdot,t))I_{q-1}(g(\cdot,t))]dt \\
&=& p^2(p-1)!\int_0^\infty \langle f(\cdot,t),g(\cdot,t)\rangle_{L^2(\R_+^{p-1})} dt \\
&=& pp!\langle f,g \rangle_{L^2(\R_+^p)}.
\end{eqnarray*}
So (\ref{int-by-parts-L}) also holds in this case. This completes the proof of (\ref{int-by-parts-L}).
\medskip

Since $\LL^{-1}(F) \in \text{Dom}(\LL) = \D^{2,2} \subset \D^{1,2}$, for any $F\in \D^{1,2}$ the quantity $\langle DF,-D\LL^{-1} F\rangle_{L^2(\R_+)}$ is well defined. As we can see in the next section, $\langle DF,-D\LL^{-1} F\rangle_{L^2(\R_+)}$ plays a key role in the normal approximation for functionals of Gaussian processes.
\medskip

In this section, we have only presented those aspects of Malliavin calculus that will be needed for our exposition of the work of Nourdin and Peccati in this paper. An extensive treatment of Malliavin calculus can be found in the book by Nualart \cite{N06}.

\section{Connecting Stein's method with Malliavin calculus}

As is {discussed} in Section 2, the Stein operator $L$ for normal approximation is given by $Lf(w)= f'(w) - wf(w)$ and the equation 
\begin{equation}\label{Stein-id-normal}
\E\big[f'(W) - Wf(W)\big] = 0
\end{equation}
holds for all $f \in C_B^1$ if and only if $W \sim \mathcal{N}(0,1)$. It is also remarked there that if $W \sim \mathcal{N}(0,1)$, (\ref{Stein-id-normal}) is a simple consequence of integration by parts. Since there is the integration by parts formula of Malliavin calculus for functionals of general Gaussian processes, there is a natural connection between Stein's method and Malliavin calculus. Indeed, integration by parts has been used in less general situations to construct the equation 
\begin{equation}\label{Stein-general2}
\E [Wf(W)] = E[T f'(W)]
\end{equation}
which is a special case of (\ref{Stein-general}). We provide two examples below.
\medskip

{\it Example 1.} Assume $\E [W] = 0$ and $\Var(W) = 1$. Then we have $\E [T] = 1$.
If $W$ has a density $\rho > 0$ with respect to the Lebesgue measure, then by integration by parts, $W$ satisfies (\ref{Stein-general2}) with $T = h(W)$, where 

$$
h(x) = \frac{\int_x^\infty y\rho(y)dy}{\rho(x)}.
$$
If $\rho$ is the density of $\mathcal{N}(0,1)$, then $h(x)= 1$ and (\ref{Stein-general2}) reduces to (\ref{Stein-id-normal}).
\medskip

{\it Example 2.} Let $X = (X_1, \dots, X_d)$ be a vector of independent Gaussian random 
variables and let $g: \R^d \rightarrow  \R$ be an absolutely continuous function. Let $W = g(X)$. Chatterjee in \cite{C09} used Gaussian interpolation and integration by parts to show that $W$ satisfies (\ref{Stein-general2}) with $T = h(X)$ where
$$
h(x) = \int_0^1 \frac{1}{2\sqrt{t}}\E \big[\sum_{i=1}^d \frac{\partial g}{\partial x_i}(x)\frac{\partial g}{\partial x_i}(\sqrt t x + \sqrt{1 - t}X)\big]dt.
$$
If $d = 1$ and $g$ the identity function, then $W \sim \mathcal{N}(0,1)$,  $h(x)=1$, and again (\ref{Stein-general2}) reduces to (\ref{Stein-id-normal}).
\medskip

As the previous example  shows (see Chatterjee \cite{C09} for details), it is possible to construct the function $h$ when one deals with sufficiently smooth functionals of a Gaussian vector. This is part of a general phenomenon discovered by Nourdin and Peccati in \cite{NP09}. Indeed, consider a functional $F$ of an isonormal Gaussian process $X=\{X(h), h\in L^2(\R_+)\}$ over $L^2(\R_+)$. Assume $F \in \D^{1,2}$, $\E [F] = 0$ and $\Var(F) = 1$.  Let $f: \R \rightarrow \R$ be  a bounded $\mathcal{C}^1$ function having a bounded derivative. Since $\LL^{-1} F \in \text{Dom}(\LL)= \mathbb{D}^{2,2}$, by (\ref{LL-inverse}) and  $\E [F] = 0$, we have 
$$
F = \LL\LL^{-1}F.
$$
Therefore, by the integration by parts formula (\ref{int-by-parts-L}), 
$$
\E [Ff(F)]=\E[\LL\LL^{-1}F \times f(F)] 
= \E[\langle Df(F), -D\LL^{-1}F \rangle_{L^2(\R_+)}]
$$
and by the chain rule,
$$
\E[\langle Df(F),-D\LL^{-1}F\rangle_{L^2(\R_+)}] = \E[f'(F)\langle DF,-D\LL^{-1}F\rangle_{L^2(\R_+)}].
$$
Hence
\begin{equation}\label{Stein-id-inner-prod}
\E [F f(F)] = \E [f'(F)\langle DF,-D\LL^{-1}F\rangle_{L^2(\R_+)}]
\end{equation}
and $F$ satisfies (\ref{Stein-general2}) with $T = \langle DF,-D\LL^{-1}F\rangle_{L^2(\R_+)}$.
\medskip

If $F$ is standard normal, that is, $F = I(\psi) =\int_0^\infty \psi dB_t$ where $\psi = I_{[0,1]}$. Then $DF = I_{[0,1]}$ and by (\ref{inverse-DiagoL}), $\LL^{-1}F = - I(\psi) = - F$. So
\begin{equation}\label{Stein-inner-prod}
\langle DF,-D\LL^{-1}F\rangle_{L^2(\R_+)} = \langle I_{[0,1]}, DF \rangle_{L^2(\R_+)} = \langle I_{[0,1]}, I_{[0,1]} \rangle_{L^2(\R_+)} = 1.
\end{equation}
This and ({\ref{Stein-id-inner-prod}) give
$$
\E Ff(F) = \E f'(F),
$$
which is the characterization equation for the standard normal distribution. 
\medskip

Now let $f_h$ be the unique bounded solution of the Stein equation (\ref{Stein-equation}) where $h: \R \rightarrow \R$ is continuous and $|h| \le 1$. Then $f_h \in \mathcal{C}^1$ and $\|f'_h\|_{\infty} \le 4\|h\|_{\infty} \le 4$, and we have
\begin{eqnarray*}
\E [h(F)] - \E [h(Z)] &=& \E \{f_h'(F)[1 - \langle DF,-D\LL^{-1}F\rangle_{L^2(\R_+)}]\} \\
&=& \E \{f_h'(F)[1 - \E (\langle DF,-D\LL^{-1}F\rangle_{L^2(\R_+)}| F)]\}.
\end{eqnarray*}
Therefore
\begin{eqnarray*}
\sup_{h \in \mathcal{C}, |h| \le 1} |\E [h(F)] - \E [h(Z)]| &\le& \|f'_h\|_{\infty}\E \Big[|1 - \E (\langle DF,-D\LL^{-1}F\rangle_{L^2(\R_+)} | F)|\Big] \\
&\le& 4\E \Big[|1 - \E (\langle DF,-D\LL^{-1}F\rangle_{L^2(\R_+)} | F)|\Big].
\end{eqnarray*}

It follows that
\begin{eqnarray*}
d_{\mathrm{TV}}(\law(F), \mathcal{N}(0,1)) &:=& \frac{1}{2}\sup_{|h|\le 1}|\E [h(F)] - \E [h(Z)]| \\
&=& \frac{1}{2}\sup_{h \in \mathcal{C}, |h| \le 1} |\E [h(F)] - \E [h(Z)]|\\ 
&\le& 2\E \Big[|1 - \E (\langle DF,-D\LL^{-1}F\rangle_{L^2(\R_+)} | F)|\Big].
\end{eqnarray*}
If, in addition, $F \in \mathbb{D}^{1,4}$, then $\langle DF,-D\LL^{-1}F\rangle_{L^2(\R_+)}$ is square-integrable and 
\begin{equation*}
\E \Big[|1 - \E (\langle DF,-D\LL^{-1}F\rangle_{L^2(\R_+)} | F)|\Big] \le \sqrt{\Var[\E (\langle DF,-D\LL^{-1}F\rangle_{L^2(\R_+)}| F)]}.
\end{equation*}

Thus we have the following theorem of Nourdin and Peccati \cite{NP09}.

\begin{thm}\label{bound-TV}
Let $F \in \mathbb{D}^{1,2}$ such that $\E [F] = 0$ and $\Var(F)=1$. Then
\begin{equation}\label{boundtv}
d_{\mathrm{TV}}(\law(F), \mathcal{N}(0,1)) \le 2\E \Big[|1 - \E (\langle DF,-D\LL^{-1}F \rangle_{L^2(\R_+)} | F)|\Big].
\end{equation}
If, in addition, $F \in \mathbb{D}^{1,4}$, then
\begin{equation}\label{boundtvwiener}
d_{\mathrm{TV}}(\law(F), \mathcal{N}(0,1)) \le 2\sqrt{\Var[\E (\langle DF,-D\LL^{-1}F\rangle_{L^2(\R_+)} | F)]}.
\end{equation}
\end{thm}

If $F$ is standard normal, (\ref{Stein-inner-prod}) implies that the upper bound in (\ref{boundtv}) is zero. This shows that the bound is tight.

\section{The fourth moment theorem}
\subsection{The fourth moment phenomenon}

The so-called fourth moment phenomenon was first discovered by Nualart and Peccati \cite{NP05} who proved that for a sequence of multiple Wiener-It\^o integrals $\{F_n\}$ of fixed order such that $\E[F_n^2] \rightarrow 1$, the following are equivalent. \\
\\
(i) $\law(F_n) \rightarrow \mathcal{N}(0,1)$; \\
(ii) $\E[F_n^4] \rightarrow 3$.
\medskip

Combining Stein's method with Malliavin calculus, Nourdin and Peccati \cite{NP09} obtained an elegant bound on the rate of convergence, which we will call {\it{the fourth moment theorem}}.

\begin{thm}\label{4th-moment-thm}
Let $F$ belong to the $k$th Wiener chaos of $B$ for $k \ge 2$ such that $\E [F^2] = 1$. Then
\begin{equation}\label{4th-moment-bd}
d_{\mathrm{TV}}(\law(F), \mathcal{N}(0,1))\le 2\sqrt{\frac{k-1}{3k}}\sqrt{\E[F^4] - 3}.
\end{equation}
\end{thm}

\begin{proof}
This proof is taken from Nourdin \cite{N12}. Write $F = I_k(f_k)$ where $f_k \in L^2_s(\R^k)$ is symmetric. By (\ref{norm-expansion}), $\E[F^2] = k!\|f_k\|^2_{L^2(\R_+)}$. By the equation (\ref{dtf}), we have $D_tF =D_tI_k(f_k)=k I_{k-1}(f_k(\cdot,t))$. Applying the product formula (\ref{prodfor}) for multiple integrals, we obtain
\begin{eqnarray}\label{DF-norm2} \notag
\frac{1}{k}\|DF\|^2_{L^2(\R_+)} &=& \frac{1}{k}\langle D F, D F\rangle_{L^2(\R_+)} \\ \notag
&=& k \int_0^\infty I_{k-1}(f_k(\cdot,t))^2 dt\\ \notag
&=&k \int_0^\infty \sum_{r=0}^{k-1}r! \binom{k-1}{r}^2 I_{2k-2-2r}\big[f_k(\cdot,t)\widetilde{\otimes}_r f_k(\cdot,t)\big]dt\\ \notag
&=&k\sum_{r=0}^{k-1}r! \binom{k-1}{r}^2 I_{2k-2-2r}\left(\int_0^\infty f_k(\cdot,t)\widetilde{\otimes}_r f_k(\cdot,t)dt\right)\\ \notag
&=&k\sum_{r=0}^{k-1}r! \binom{k-1}{r}^2 I_{2k-2-2r}\left(f_k\widetilde{\otimes}_{r+1} f_k\right)\\ \notag
&=&k \sum_{r=1}^{k} (r-1)! \binom{k-1}{r-1}^2 I_{2k-2r}\left(f_k\widetilde{\otimes}_r f_k\right)\\ \notag
&=&k \sum_{r=1}^{k-1} (r-1)! \binom{k-1}{r-1}^2I_{2k-2r}\left(f_k\widetilde{\otimes}_r f_k\right)+ k!\Vert f\Vert_{L^2(\R_+)}^2 \\ 
&=&k \sum_{r=1}^{k-1} (r-1)! \binom{k-1}{r-1}^2I_{2k-2r}\left(f_k\widetilde{\otimes}_r f_k\right)+ \E[F^2].
\end{eqnarray}
Note that since  $F = I_k(f_k)$ and $\E[F] = 0$, we have $\LL^{-1}F = -\frac{1}{k}F$. So 
$$
\langle DF,-D \LL^{-1} F\rangle_{L^2(\R_+)} = \frac{1}{k}\langle D F, D F\rangle_{L^2(\R_+)} = \frac{1}{k}\|DF\|^2_{L^2(\R_+)}.
$$
Letting $f(F) = F$ in the Stein identity (\ref{Stein-id-inner-prod}), we obtain
$$
\E \big[\langle DF,-D \LL^{-1} F\rangle_{L^2(\R_+)}\big]  = \E[F^2].
$$
Applying the orthogonality of the Wiener chaos and the formula (\ref{ortho-relation}),
\begin{equation} \notag
\Var\left(\langle DF,-D \LL^{-1} F\rangle_{L^2(\R_+)}\right) 
= \Var \left(\frac{1}{k}\|DF\|^2_{L^2(\R_+)}\right)  
\end{equation}
\begin{equation}  \label{Var-DF-norm2} 
=\sum_{r=1}^{k-1}\frac{r^2}{k^2}(r!)^2\binom{k}{r}^4(2k-2r)!\Vert f_k\widetilde{\otimes}_r f_k\Vert_{L^2(\R^{2k-2r})}^2.  
\end{equation}
By the product formula (\ref{prodfor}) again, we have
\begin{equation} \label{F2}
F^2=\sum_{r=0}^k r! \binom{k}{r}^2 I_{2k-2r}(f_k\widetilde{\otimes}_r f_k).
\end{equation}
Applying the Stein identity (\ref{Stein-id-inner-prod}), we have 
\begin{eqnarray}\label{EF4} \notag
\E[F^4] = \E[F \times F^3] &=& 3\E \big[F^2 \times \langle DF,-D\LL^{-1}F\rangle_{L^2(\R_+)}\big] \\ 
&=& 3\E \big[F^2 \times \frac{1}{k}\|DF\|^2_{L^2(\R_+)}\big].
\end{eqnarray}
This together with (\ref{DF-norm2}), (\ref{F2}) and the formula (\ref{ortho-relation}) yield
\begin{eqnarray}\label{ExpF4} \notag
\E[F^4]&=& 3(\E[F^2])^2+\frac{3}{k}\sum_{r=1}^{k-1}r (r!)^2 \binom{k}{r}^4 (2k-2r)!\Vert f_k \widetilde{\otimes}_r f_k \Vert^2_{L^2(\R^{2k-2r}_+)}\\
&=& 3 + \frac{3}{k}\sum_{r=1}^{k-1}r (r!)^2 \binom{k}{r}^4(2k-2r)!\Vert f_k \widetilde{\otimes}_r f_k \Vert^2_{L^2(\R^{2k-2r}_+)}.
\end{eqnarray}
Comparing (\ref{Var-DF-norm2}) and (\ref{ExpF4}) leads to 
\begin{equation}\label{Var-ineq}
\Var\left(\langle DF,-D \LL^{-1} F\rangle_{L^2(\R_+)} \right) \le \frac{k-1}{3k}\left(\E[F^4] - 3 \right).
\end{equation}
Since $\Var[\E(\langle DF,-D\LL^{-1}F\rangle_{L^2(\R_+)} | F)] \le \Var\left(\langle DF,-D \LL^{-1} F\rangle_{L^2(\R_+)} \right)$, Theorem \ref{4th-moment-thm} follows from (\ref{boundtvwiener}).
\end{proof}
As one can see from (\ref{ExpF4}), $\E[F^4] \ge 3$ whenever $F$ is a multiple Wiener-It\^o integral with variance $1$. Theorem \ref{4th-moment-thm} also implies the result of Nualart and Peccati \cite{NP05} mentioned above. Without loss of generality, we assume that $\E[F_n^2] = 1$. The part of (ii) $\Longrightarrow$ (i) follows immediately from (\ref{4th-moment-bd}). For the part of (i) $\Longrightarrow$ (ii) (which actually is independent of Theorem \ref{4th-moment-thm}), we observe that by the continuous mapping theorem, we have $\law(F_n^4) \rightarrow \law(Z^4)$ where $Z\sim \mathcal{N}(0,1)$. Write $F_n = I_k(f_n)$. By the hypercontractivity inequality (Nelson \cite{N73}),
\begin{equation} \notag
\E[|I_k(f)|^r] \le [(r-1)^kk!]^r\|f\|^r_{L^2(\R^k_+)} \quad \text{for} \quad k\ge 1, r \ge 2, f \in L^2(\R_+^k),
\end{equation}
and the given condition that $k!\|f_n\|^2_{L^2(\R^k_+)} = \E[I_k(f_n)^2]=\E[F_n^2] = 1$, we have $\sup_n \E[|F_n|^r] < \infty$ for $r > 4$. This implies that $\{F_n^4\}$ is uniformly integrable and therefore $\E[F_n^4] \rightarrow \E[Z^4] = 3$, and (ii) follows.
\medskip

From (\ref{ExpF4}), we observe that (ii) is equivalent to $\Vert f_k \widetilde{\otimes}_r f_k \Vert^2_{L^2(\R^{2k-2r}_+)} \rightarrow 0$ for $r = 1,\cdots,k-1$. This fact is also contained in the theorem of Nualart and Peccati \cite{NP05}. The equation (\ref{ExpF4}) also shows that the calculation of $\E[F^4] - 3$ depends on that of $\Vert f_k \widetilde{\otimes}_r f_k \Vert^2_{L^2(\R^{2k-2r}_+)}$ for $r = 1,\cdots,k-1$.
\medskip

In more recent work, Nourdin and Peccati \cite{NP13} proved the following {\it{optimal fourth moment theorem}}, which improves Theorem \ref{4th-moment-thm}. 

\begin{thm}\label{optimal-4th-moment-thm}
Let $\{F_n\}$ be a sequence of random variables living in a Wiener chaos of fixed order such that $\E[F_n^2] =1$. Assume that $F_n$ converges to $Z \sim \mathcal{N}(0,1)$, in which case $\E[F_n^3] \rightarrow 0$ and $\E[F_n^4]\rightarrow 3$. Then there exist two finite constants, $0 < c < C$, possibly depending on the order of the Wiener chaos and on the sequence $\{F_n\}$, but not on $n$, such that 
\begin{equation}\label{optimal-4th-moment-bd}
cM(F_n)\le d_{\mathrm{TV}}(\law(F_n), \mathcal{N}(0,1))\le CM(F_n),
\end{equation}
where $M(F_n)=\max\{\E[F_n^4]-3,|\E[F_n^3]|\}$. 
\end{thm}
This shows that 
the bound in (\ref{4th-moment-bd}) is optimal if and only if $\sqrt{\E[F_n^4]-3}$ and $|\E[F_n^3]|$ are of the same order (typically $\frac{1}{\sqrt{n}}$).

\subsection{Breuer-Major theorem}
In this subsection, we show how the fourth moment theorem, that is, Theorem \ref{4th-moment-thm}, can be applied to prove the Breuer-Major theorem \cite{BM83}. We begin by first introducing the notion of {\it Hermite rank} of a function. It is well-known that every $\phi \in L^2\big(\R, \frac{1}{\sqrt{2\pi}}e^{-x^2/2}dx\big)$ can be expanded in a unique way in terms of the Hermite polynomials as follows.
\begin{equation} \label{phi-Hermite}
\phi(x) = \sum_{q=0}^\infty a_qH_q(x).
\end{equation}
We call $d$ the Hermite rank of $\phi$ if $d$ is the first integer $q \ge 0$ such that $a_q \neq 0$. We now state the Breuer-Major theorem.
\begin{thm} \label{thm-B-M}
Let $\{X_k\}_{k\ge 1}$ be a centered stationary Gaussion sequence, where each $X_k \sim \mathcal{N}(0,1)$, and let $\phi \in L^2\big(\R, \frac{1}{\sqrt{2\pi}}e^{-x^2/2}dx\big)$ be given by (\ref{phi-Hermite}). Assume that $a_0 = \E[\phi(X_1)] = 0$ and that $\sum_{k\in\Z}|\rho(k)|^d < \infty$, where $\rho$ is the covariance function of $\{X_k\}_{k\ge 1}$ and $d$ the Hermite rank of $\phi$. Let $V_n = \frac{1}{\sqrt{n}}\sum_{k=1}^n \phi(X_k)$. Then as $n \rightarrow \infty$, we have
\begin{equation}\label{conv-in-law}
\law(V_n) \rightarrow \mathcal{N}(0,\sigma^2)
\end{equation}
where $\sigma^2 \in [0,\infty)$ and is given by 
\begin{equation}\label{sigma^2}
\sigma^2 = \sum_{q=d}^\infty q!a_q^2\sum_{k\in \Z}\rho(k)^q.
\end{equation}
\end{thm}
The original proof of Theorem \ref{thm-B-M} uses the method of moments, by which one has to compute all the moments of $V_n$ and show that they converge to the corresponding moments of the limiting distribution. The fourth moment theorem offers a much simpler approach by which we only need to deal with the fourth moment of $V_n$. We will give a sketch of the proof here that applies the fourth moment theorem. A detailed proof can be found in Nourdin \cite{N12}. 

\begin{proof}

First we show that
\begin{equation}\label{4th-moment-V_n}
\Var(V_n) = \E[V_n^2] = \sum_{q=d}^\infty q!a_q^2\sum_{r\in \Z}\rho(r)^q(1 - \frac{|r|}{n})I(|r|< n).
\end{equation}
Since
$$
q!a_q^2|\rho(r)|^q(1 - \frac{|r|}{n})I(|r|< n)\le q!a_q^2|\rho(r)|^q \le q!a_q^2|\rho(r)|^d 
$$
and
$$
\sum_{q=d}^\infty\sum_{r\in\Z}q!a_q^2|\rho(r)|^d = \E[\phi^2(X_1)]\sum_{r\in\Z}|\rho(r)|^d < \infty,
$$
it follows by an application of the dominated convergence theorem that $\E[V_n^2]\rightarrow \sigma^2$, where $\sigma^2 \in [0,\infty)$ and is given by (\ref{sigma^2}). If $\sigma^2 = 0$, then there is nothing to prove. So we assume that $\sigma^2 > 0$. 
\medskip

The proof of (\ref{conv-in-law}) can be divided into three parts in increasing generality of $\phi$: (i) $\phi$ is a Hermite polynomial, (ii) $\phi$ is a real polynomial, and (iii) $\phi \in L^2\big(\R,\frac{1}{\sqrt{2\pi}}e^{-x^2/2}dx)$.  We sketch the proof of part (i). Let $\mathcal{H}$ be the real separable Hilbert space generated by $\{X_k\}_{k\ge 1}$ and let $\psi: \mathcal{H}\rightarrow L^2(\R_+)$ be an isometry. Define $h_k = \psi(X_k)$ for $k \ge 1$. Then we have
$$
\int_0^\infty h_k(x)h_l(x)dx = \E[X_kX_l] = \rho(k-l).
$$
Therefore 
$$
\law\{X_k: k \in \N\} = \law\big\{\int_0^\infty h_k(t)dB_t: k \in \N\big\},
$$
where $B = (B_t)_{t\ge0}$ is a standard Brownian motion. Note that for each $k \ge 1$, $\|h_k\|_{L^2(\R_+)}^2 = \E[X_k^2] = 1$. Since $\phi = H_q$ for some $q \ge 1$, we have
\begin{eqnarray*}
V_n = \frac{1}{\sqrt{n}}\sum_{k=1}^nH_q(X_k)&\stackrel{\law}{=}& \frac{1}{\sqrt{n}}\sum_{k=1}^nH_q\big(\int_0^\infty h_k(t)dB_t\big) \\
&=& \frac{1}{\sqrt{n}}\sum_{k=1}^n I_q(h_k^{\otimes q}) = I_q(f_{n,q})
\end{eqnarray*}
where 
$$
f_{n,q} = \frac{1}{\sqrt{n}}\sum_{k=1}^n h_k^{\otimes q}.
$$
It can be shown (see Nourdin \cite{N12} for details) that 
$\Vert f_{n,q} \widetilde{\otimes}_r f_{n,q} \Vert^2_{L^2(\R^{2k-2r}_+)} \rightarrow 0$ as $n \rightarrow \infty$ for $r = 1,\cdots,k-1$. By Theorem \ref{4th-moment-thm} and (\ref{ExpF4}) taking into account an appropriate scaling, part (i) is proved. Part (ii) follows from part (i) by writing a polynomial as a linear combination of Hermite polynomials and then applying a theorem of Peccati and Tudor \cite{PT05}, which concerns the equivalence between marginal and joint convergence in distribution of multiple Wiener-It\^o integrals to the normal distributions. For part (iii), write 
\begin{eqnarray*}
V_n &=& \frac{1}{\sqrt{n}}\sum_{k=1}^n \sum_{q=1}^N a_qH_q(X_k) + \frac{1}{\sqrt{n}}\sum_{k=1}^n \sum_{q=N+!}^\infty a_qH_q(X_k) \\
&=& V_{n,N} + R_{n,N}.
\end{eqnarray*}
Then apply part (ii) to $V_{n,N}$ and show that $\sup_{n\ge 1}\E[R_{n,N}^2] \rightarrow 0$ as $N \rightarrow \infty$. This completes the proof of Theorem \ref{thm-B-M}.
\end{proof}
Bounds on the rate of convergence in the Breuer-Major theorem have been obtained by Nourdin, Peccati and Podoskij \cite{NPP11}, who considered random variables of the form $S_n = \frac{1}{\sqrt{n}}\sum_{k=1}^n [f(X_k) - \E f(X_x)]$, $n \ge 1$, where $\{X_k\}_{k\in \Z}$ is a d-dimensional stationary Gaussian process and $f: \R^d \rightarrow \R$ a measurable function. They obtained explicit bounds on $|\E h(S_n)-\E h(S)|$, where $S$ is a normal random variable and $h$ a sufficiently smooth function. Their results both generalize and refine the Breuer-Major theorem and some other central limit theorems in the literature. The methods they used are based on Malliavin calculus, interpolation techniques and Stein's method.

\subsection{Quadratic variation of fractional Brownian motion}

In this subsection, we consider another application of Theorem \ref{4th-moment-thm} and also of Theorem \ref{optimal-4th-moment-thm}. Let $B^H = (B_t^H)_{t\ge0}$ be a fractional Brownian motion with Hurst index $H \in (0,1)$, that is, $B^H$ is a centered Gaussian process with covariance function given by
$$
\E[B_t^HB_s^H] = \frac{1}{2}(t^{2H}+s^{2H}- |t-s|^{2H}),\quad s,t \ge 0.
$$
This $B^H$ is self-similar of index $H$ and has stationary increments. 
\medskip

Consider the sum of squares of increments,
\begin{equation}\label{qvF}
F_{n,H} = \frac{1}{\sigma_n}\sum_{k=1}^n[(B_k^H - B_{k-1}^H)^2-1] = \frac{1}{\sigma^n}\sum_{k=1}^n H_2(B_k^H - B_{k-1}^H)
\end{equation}
where $H_2$ is the 2nd Hermite polynomial and $\sigma_n >0$ is such that $\E[F_{n,H}^2] =1$.
An application of the Breuer-Major theorem shows that for $0 < H \le \frac{3}{4}$,
$$
\law(F_{n,H}) \rightarrow \mathcal{N}(0,1) \quad \text{as}\quad n \rightarrow \infty.
$$
Nourdin and Peccati \cite{NP12} applied Theorem \ref{4th-moment-thm} to prove the following theorem which provides the rates of convergence for different values of the Hurst index $H$.
\begin{thm} \label{qv-bd-thm}
Let $F_{n,H}$ be as defined in (\ref{qvF}). Then
\begin{equation} \label{qv-bd}
d_{\mathrm{TV}}(F_{n,H},\mathcal{N}(0,1))\le c_H
\left\{\begin{array}{lrl}
\frac{1}{\sqrt{n}}&\text{if} & H\in(0,\frac 5 8)\\
\frac{(\log n)^{\frac 3 2}}{\sqrt{n}}&\text{if}& H=\frac 5 8\\
n^{4H-3}&\text{if}& H\in (\frac 5 8,\frac 3 4)\\
\frac {1}{\log n}&\text{if}& H=\frac 3 4.
\end{array}
\right.
\end{equation}
\end{thm}
\begin{proof}
We will give a sketch of the proof in Nourdin \cite{N12}. Consider the closed linear subspace $\mathcal{H}$ of $L^2(\Omega)$ generated by $(B_{k}^H)_{k\in \N}$.
As it is a real separable Hilbert space, there exists an isometry $\psi: \mathcal{H} \rightarrow L^2(\R_+)$. For any $k \in \N$, define $h_k = \psi(B_k^H-B_{k-1}^H)$. Then for $k,l\in \N$, we have
\begin{equation}\label{isometry}
\int_0^\infty h_k(x)h_l(x)dx = \E[(B_k^H-B_{k-1}^H)(B_l^H-B_{l-1}^H)] = \rho(k-l)
\end{equation}
where
\begin{equation}\label{rho}
\rho(r) = \frac{1}{2}(|r+1|^{2H}+|r-1|^{2H}-2|r|^{2H}).
\end{equation}
Therefore
$$
\law\{B_k^H-B_{k-1}^H:k\in \N\}= \law \Big\{\int_0^\infty h_k(t)dB_t: k\in \N \Big\}
$$
where $B=(B_t)_{t\ge 0}$ is a standard Brownian motion. Consequently, without loss of generality, we can regard $F_{n,H}$ as
\begin{equation}\notag
F_n = \frac{1}{\sigma_n}\sum_{k=1}^n H_2\big(\int_0^\infty h_k(t)dB_t\big).
\end{equation}
Since for $k \in \N$, $\|h_k\|_{L^2(\R_+)}^2 = \rho(0) = 1$ (by (\ref{isometry}) and (\ref{rho})), we have 
\begin{equation}
F_n = \frac{1}{\sigma_n}\sum_{k=1}^n I_2(h_k\otimes h_k)= I_2(f_n)
\end{equation}
where $I_p$, $p \ge 1$, is the $p$th multiple Wiener-It\^o integral with respect to $B$, and 
$$
f_n = \frac{1}{\sigma_n}\sum_{k=1}^nh_k\otimes h_k.
$$

Now straightforward calculations yield
$$
\sigma_n^2 = 2 \sum_{k,l=1}^n \rho^2(k-l) = 2 \sum_{|r|<n}(n-|r|)\rho^2(r).
$$
It can be shown that for $H< \frac{3}{4}$, we have $\sum_{r \in \Z}\rho^2(r) < \infty$, and 
\begin{equation}\label{var-limit1}
\lim_{n \rightarrow \infty} \frac{\sigma^2_n}{n} = 2\sum_{r \in \Z}\rho^2(r),
\end{equation}
and for $H=\frac{3}{4}$, we have
\begin{equation}\label{var-limit2}
\lim_{n \rightarrow \infty} \frac{\sigma^2_n}{n\log n}=\frac{9}{16}.
\end{equation}

Now we come to calculating the bound $\sqrt{\E[F_n^4] - 3}$ in Theorem \ref{4th-moment-thm}. We first note that $f_n$ is symmetric, and so $f_n\widetilde{\otimes}f_n = f_n\otimes f_n$. Therefore, by (\ref{ExpF4}), we have
\begin{eqnarray} \label{ExpF4-3} \notag
\E[F_n^4] - 3 &=& 48\|f_n \widetilde{\otimes}_1f_n\|_{L^2(\R_+^2)}^2 \\ \notag
&=& 48\|f_n \otimes_1 f_n\|_{L^2(\R_+^2)}^2 \\ 
&=& \frac{48}{\sigma_n^4}\sum_{i,j,k,l=1}^n\rho(k-l)\rho(i-j)\rho(k-i)\rho(l-j).
\end{eqnarray}
By bounding the extreme right of (\ref{ExpF4-3}) (see Nourdin \cite{N12} for details), we obtain
\begin{equation}\label{ExpF4-3-bd}
\E[F_n^4]-3 \le \frac{48n}{\sigma_n^4}\left(\sum_{|k|<n}|\rho(k)|^{\frac{4}{3}}\right)^3.
\end{equation}
From the asymptotic behavior of $\rho(k)$ as $|k| \rightarrow \infty$, we can show that
\begin{equation}\label{rho-limit}
\sum_{|k|<n}|\rho(k)|^{\frac{4}{3}} = 
\left\{\begin{array}{lrl}
O(1)&\text{if} & H\in(0,\frac 5 8)\\ 
O(\log n)&\text{if}& H=\frac 5 8\\
O(n^{(8H-5)/3} &\text{if}& H\in (\frac 5 8,1).
\end{array}\right.
\end{equation}
This, together with (\ref{var-limit1}) and (\ref{ExpF4-3-bd}), implies
\begin{equation}\notag
\sqrt{\E[F_n^4] - 3} \le c_H
\left\{\begin{array}{lrl}
\frac{1}{\sqrt{n}}&\text{if} & H\in(0,\frac 5 8)\\
\frac{(\log n)^{3/2}}{\sqrt{n}}&\text{if}& H=\frac 5 8\\
n^{(4H-3)} &\text{if}& H\in (\frac 5 8,\frac{3}{4}).
\end{array}\right.
\end{equation}

For $H=\frac{3}{4}$, combining (\ref{var-limit2}), (\ref{ExpF4-3-bd}) and (\ref{rho-limit}) gives
\begin{equation} \notag
\sqrt{\E[F_n^4] - 3} = O\left(\frac{1}{\log n}\right).
\end{equation}

This proves Theorem \ref{qv-bd-thm}
\end{proof}

In Nourdin and Peccati \cite{NP13}, the bounds in (\ref{optimal-4th-moment-bd}) are applied to obtain the following improvement of (\ref{qv-bd}) for $H\in(0,\frac{3}{4})$.

\begin{thm}
Let $F_{n,H}$ be as defined in (\ref{qvF}). Then
\begin{equation} \notag
d_{\mathrm{TV}}(F_{n,H},\mathcal{N}(0,1))\propto
\left\{\begin{array}{lrl}
\frac{1}{\sqrt{n}}&\text{if} & H\in(0,\frac{2}{3})\\
\frac{(\log n)^2}{\sqrt{n}}&\text{if}& H=\frac{2}{3}\\
n^{6H-\frac{9}{2}}&\text{if}& H\in (\frac{2}{3},\frac{3}{4}).
\end{array}
\right.
\end{equation}
where for nonnegative sequences $(u_n)$ and $(v_n)$, we write $v_n \propto u_n$ to mean $0 < \liminf v_n/u_n \le \limsup v_n/u_n < \infty$.
\end{thm}
For $H>\frac 3 4$, $F_{n,H}$ does not converge to a Gaussian distribution. Instead, it converges to the so-called Rosenblatt distribution, which belongs to the second Wiener chaos and is therefore not Gaussian.
\medskip

The expository paper by Nourdin \cite{N12}, the survey paper by Peccati \cite{P14} with an emphasis on more recent results, and the book by Nourdin and Peccati \cite{NP12}, cover many topics and give detailed development of this new area of normal approximation.

\section{Acknowlegment}
I would like to thank Ivan Nourdin for some very helpful discussions during the course of writing this paper and for reading the drafts of this paper and giving very helpful comments. This work is partially supported by Grant C-146-000-034-001 and Grant R-146-000-182-112 from the National University of Singapore.

%
%
%

\end{document}